\documentclass[11pt,twoside,a4paper]{article}
\usepackage{amsmath, amsthm, amssymb, enumerate}
\usepackage{url}
\usepackage{fullpage}
\usepackage{stmaryrd}
\usepackage{hyperref}
\usepackage{amscd}
\usepackage{booktabs}
\usepackage{microtype}
\usepackage[nottoc]{tocbibind}
\usepackage[all, cmtip]{xy}

\theoremstyle{plain}
\newtheorem{thm}{Theorem}[section]
\newtheorem{lemma}[thm]{Lemma}
\newtheorem{prop}[thm]{Proposition}
\newtheorem{cor}[thm]{Corollary}
\theoremstyle{definition}

\theoremstyle{remark}

\numberwithin{equation}{section}

\def \bF {\mathbb F}

\def \bQ {\mathbb Q}

\def \bZ {\mathbb Z}

\def \fp {\mathfrak p}

\begin{document}

\begin{center}
\Large
\textbf{Distinguishing eigenforms modulo a prime ideal}

\large
\begin{tabular}{cc}
Sam Chow & Alex Ghitza \\
University of Bristol & University of Melbourne
\end{tabular}

\end{center}

\section*{Abstract}
\setcounter{page}{1}

Consider the Fourier expansions of two elements of a given space of modular
forms. How many leading coefficients must agree in order to guarantee that the
two expansions are the same? Sturm~\cite{Sturm} gave an upper bound for
modular forms of a given weight and level. This was adapted by Ram
Murty~\cite{Murty}, Kohnen~\cite{Kohnen} and Ghitza~\cite{Ghitza} to the case of two eigenforms of
the same level but having potentially different weights. We consider their
expansions modulo a prime ideal, presenting a new bound. In the process of
analysing this bound, we generalise a result of Bach and Sorenson~\cite{Bach},
who provide a practical upper bound for the least prime in an arithmetic
progression.

\section*{Notation and terminology}
All modular forms discussed are of positive integer weight $k$ and level $N$. 
A modular form of weight $k$ 
and character $\chi$ for $\Gamma_0(N)$ satisfies
\begin{equation}
  f\left(\frac{az+b}{cz+d}\right)=\chi(d)(cz+d)^k f(z)\quad\text{for all }
  \begin{pmatrix}a&b\\c&d\end{pmatrix}\in \Gamma_0(N).
\end{equation}
By \emph{eigenform} we mean an eigenvector for the full Hecke algebra. If $f$ is a modular form then $a_n(f)$ denotes the $n$th Fourier coefficient: 
\begin{equation}
f(z) = \sum_n a_n(f) e^{2 \pi i nz}.
\end{equation}
The symbols $p$ and $\ell$ are reserved for prime numbers. A \emph{prime
  primitive root} modulo $p$ is a prime that is also a primitive root modulo
$p$. We write $f \sim g$ to mean that the ratio of the two functions tends to
1 in some limit, and define the equivalence relation $\sim$ analogously for
sequences. The Euler totient function is denoted by $\varphi$. By GRH we 
mean the generalisation of the Riemann hypothesis to Dirichlet $L$-functions.
If $a$ and $q$ 
are relatively prime positive integers and $x \ge 1$ is a real number, then 
$\pi_{a,q}(x)$ denotes the number of $\ell \le x$ such that 
$\ell \equiv a \mod q$. We use Landau `big O' notation in the standard way.

\section{Introduction}
\label{intro}

We present a new bound for the number of leading Fourier coefficients that one
needs to compare in order to distinguish two eigenforms, of potentially
different weights, modulo a prime ideal. Bounds of this flavour are of great
practical use in modular forms research, and have received much attention
(e.g. \cite{Murty}, \cite{Kohnen}, \cite{Ghitza}, \cite{GH}, \cite{Kowalski})
since the groundbreaking work of Sturm~\cite{Sturm}:

\begin{thm} [Sturm bound, see {\cite[Theorem 9.18]{Stein}}]
Let $f$ be a modular form of weight $k$ for a
congruence subgroup $\Gamma$ of index $i(\Gamma)$ inside $SL_2(\bZ)$. 
Let $R$ be the ring of integers of a number field, 
and assume that $R$ contains the
Fourier coefficients of $f$. Let $\fp$ be a prime ideal in $R$, 
and assume 
that $f \not \equiv 0 \mod \fp$.
Then there exists
\begin{equation} n \le \frac {k \cdot i(\Gamma)}{12}  \end{equation}
such that $a_n(f) \not \equiv 0 \mod \fp$.
\end{thm}

We will use Buzzard's adaptation of the Sturm bound to modular forms with
character:
\begin{cor} [see {\cite[Corollary 9.20]{Stein}}] \label{BuzzardCor} 
Let $f$ and $g$ be modular forms of weight $k$ and character $\chi$ 
for $\Gamma_0(N)$. 
Let $R$ be the ring of integers of a number field, 
and assume that $R$ contains the
Fourier coefficients of $f$ and $g$.
Let $\fp$ be a prime ideal in $R$, 
and assume that $f \not \equiv g \mod \fp$. 
Then there exists
\begin{equation} n \le \frac k{12} [SL_2(\bZ):\Gamma_0(N)] \end{equation}
such that $a_n(f) \not \equiv a_n(g) \mod \fp$.
\end{cor}

Our research is strongly motivated by work of Ram Murty~\cite{Murty}:
\begin{lemma} [see {\cite[Lemma 2]{Ghitza}}] \label{MurtyLemma}
Let $f$ and $g$ be eigenforms of respective weights $k_1 \ne k_2$ for $\Gamma_0(N)$, and let $\ell$ be the least prime not
dividing $N$. Then there exists $n \le \ell^2$ such that $a_n(f) \ne a_n(g)$.
\end{lemma}

Our main result concerns eigenforms modulo a prime ideal:

\begin{thm} \label{main}
Let $f$ and $g$ be normalised eigenforms for 
$\Gamma_0(N)$, with
character $\chi$ and respective weights $k_1 \le k_2$. 
Let $R$ be the ring of integers of a number field
containing the Fourier coefficients of $f$ and $g$,
and let $\fp$ be a nonzero prime ideal in $R$. 
Define $p$ by $p \bZ = \fp \cap \bZ$, assume that $p \ge 5$, 
and assume that $f \not \equiv g \mod \fp$. 
Then there exists
\begin{equation} \label{GhitzaBound}
n \le \max\left\{ g^*(p,N)^2, \frac{k_2}{12}[SL_2(\bZ):\Gamma_0(N)] \right\}
\end{equation}
such that $a_n(f) \ne a_n(g) \mod \fp$, where $g^*(p,N)$ 
is the least prime primitive root modulo $p$ that does not divide $N$.
\end{thm}

We note that Kohnen has obtained a similar result~\cite[Theorem 4]{Kohnen},
replacing $g^*(p,N)^2$ by the constant $900$ in~\eqref{GhitzaBound}, at the
expense of requiring $(N, 30)=1$ and only getting a bound for infinitely many
(rather than all) prime ideals $\fp$ of $R$.

Our argument can be modified to deal with the excluded cases $p=2$ and $p=3$,
yielding (slightly weaker) versions of Theorem~\ref{main}. We relegate these
special cases to Section~\ref{2and3}. In Section \ref{proof} we prove Theorem
\ref{main}. In Section \ref{asymptotics}, we provide asymptotics (as $N \to
\infty$) for the two quantities in the bound \eqref{GhitzaBound}, 
establishing that the second is asymptotically greater. In
Section \ref{practical} we determine how large $N$ has to be to ensure that
the second expression in \eqref{GhitzaBound} is indeed the larger of the two.
The crucial ingredient in Section \ref{practical} is our generalisation (see
Corollary \ref{dist}) of an explicit Linnik-type bound (see Theorem \ref{BS})
of Bach and Sorenson.

We thank James Withers for several fruitful discussions and observations. We 
thank M. Ram Murty and David Loeffler for some useful comments. The first 
author was supported by the Elizabeth and Vernon Puzey scholarship, and is 
grateful towards the University of Melbourne for their hospitality while 
preparing this memoir. The second author was supported by Discovery Grant
DP120101942 from the Australian Research Council.

\section{Proof of Theorem \ref{main}}
\label{proof}

Since $p -1 \ge 4$ is even, we may use the (appropriately normalised) Eisenstein series of weight $p-1$, which is the modular form for $SL_2(\bZ)$ given by
\begin{equation}
E_{p-1}(z) = 1- \frac{2p-2}{B_{p-1}} \sum_{n=1}^\infty \sigma_{p-2}(n)  e^{2 \pi i nz}, 
\end{equation}
where $B_{p-1}$ is the $(p-1)$st Bernoulli number (a rational number) and
$\sigma_{p-2}(n) = \sum_{d|n} d^{p-2}$; see~\cite[Subsection 2.1.2]{Stein}. 

If $k_1 = k_2$ then the result follows immediately from Corollary
\ref{BuzzardCor}, so henceforth assume that $k_1 < k_2$. Put $\ell =
g^*(p,N)$. By standard formulae (see~\cite[Proposition 5.8.5]{Diamond}), 
\begin{equation}\label{hecke} 
\chi(\ell) \ell^{k_1-1} = a_\ell(f)^2 - a_{\ell^2}(f) \qquad \text{and}
\qquad \chi(\ell) \ell^{k_2-1} = a_\ell(g)^2 - a_{\ell^2}(g). 
\end{equation}
We may assume that $a_\ell(f) \equiv a_\ell(g) \mod \fp$ and $a_{\ell^2}(f)
\equiv a_{\ell^2}(g) \mod \fp$, since otherwise the result is immediate.
As
$(\ell, N)=1$,
it follows from~\eqref{hecke} that
\begin{equation} 
  \ell^{k_1} - \ell^{k_2} \in \fp \cap \bZ = p\bZ.
\end{equation}
As $\ell$ is a primitive root modulo $p$, this implies that $p-1$ divides $k_2 - k_1$, so put
\begin{equation} r =  \frac{k_2-k_1}{p-1}\end{equation}
and $f^\prime = E_{p-1}^rf$. The von Staudt-Clausen theorem (see~\cite[Theorem 5.8.4]{BorevichShafarevich}) implies that $p$ divides the denominator of $B_{p-1}$, so 
\begin{equation} E_{p-1} \equiv 1 \mod p \end{equation}
as power series. Now $f^\prime \equiv f \mod pR$, so $f^\prime \equiv f \mod
\fp$. As $f^\prime$
is a modular form of weight $k_2$ and character $\chi$ for the congruence subgroup $\Gamma_0(N)$, the result now follows from Corollary \ref{BuzzardCor}.

\section{Asymptotics} 
\label{asymptotics}

We show that, of the two expressions in Theorem \ref{main}, the second is greater, providing that $N$ is sufficiently large. The key result in this section is:
\begin{thm} \label{key} Let $p \ge 5$. Then
\begin{equation}  \limsup_{N \to \infty} \frac{g^*(p,N)}{\log N} = \frac{p-1}{\varphi(p-1)}. \end{equation}
\end{thm}
The group index $[SL_2(\bZ):\Gamma_0(N)]$ is classically known
(see~\cite[Exercise 1.2.3]{Diamond}):
\begin{equation} \label{ShimuraFormula}
  [SL_2(\bZ):\Gamma_0(N)] = N \prod_{\ell | N} \left(1+\frac{1}{\ell}\right). 
\end{equation}
In particular $[SL_2(\bZ):\Gamma_0(N)] \ge N$ which, upon proving Theorem \ref{key}, will verify the assertion made at the beginning of this section. 

We include the supremal asymptotics for $[SL_2(\bZ):\Gamma_0(N)]$ purely for interest's sake (this is proved in a similar vein to Theorem \ref{key}):
\begin{prop} \label{side}
\begin{equation}  \limsup_{N \to \infty} \frac{[SL_2(\bZ):\Gamma_0(N)]}{N \log \log N} = \frac{6e^\gamma}{\pi^2}, \end{equation}
where $\gamma$ is the Euler-Mascheroni constant. 
\end{prop}

Our goal for the remainder of this section is to prove Theorem \ref{key}. For positive integers $t$, let $x_t$ be the $t$th smallest prime primitive root modulo $p$, and let $N_t = x_1 \cdots x_t$ (also put $N_0 =1$). The sequence $(N_t)$ is the worst case scenario: if $N$ is a positive integer then there exists $t \ge 0$ (defined by $g^*(p,N) = x_{t+1}$) such that $g^*(p,N_t) =g^*(p,N)$ and $N_t \le N$. Put
\begin{equation} c = \frac{p-1}{\varphi(p-1)} > 1.
\end{equation}

We will establish Theorem \ref{key} via the following:
\begin{prop} \label{intermediate}
\begin{equation}
\lim_{t \to \infty} \frac{x_t}{\log N_t} = c.
\end{equation}
\end{prop}
This in turn is established by determining the asymptotics of the sequence $(x_t)$:
\begin{lemma} \label{asymp}
\begin{equation} x_t \sim ct \log t. \end{equation}
\end{lemma}

We require some basic results on asymptotic equivalence:
\begin{lemma} \label{basic}
\begin{enumerate}[(i)]
\item Let $(a_t)$ and $(b_t)$ be sequences of positive real numbers. Assume that $a_t \sim b_t$ and that $b_t \to \infty$ as $t \to \infty$. Then $\log a_t \sim \log b_t$. 
\item Let $(a_t), (b_t), (c_t)$, and $(d_t)$ be sequences of positive real numbers such that $a_t \sim c_t$ and $b_t \sim d_t$. Then $a_t+b_t \sim c_t + d_t$.
\end{enumerate}
\end{lemma}

Armed with these tools, we prove Lemma~\ref{asymp},
Proposition~\ref{intermediate}, and Theorem~\ref{key}.
\begin{proof}[Proof of Lemma~\ref{asymp}]
We interpret $t$ as the number of prime primitive roots modulo $p$ that are less than or equal to $x_t$. Each of these lies in one of the $\varphi(p-1)$ primitive root residue classes, so summing the prime number theorem for arithmetic progressions over these residue classes yields
\begin{equation} \label{I} t \sim \frac{\varphi(p-1)}{p-1} \cdot \frac{x_t}{\log x_t},\end{equation}
so \begin{equation} \label{II} \log t \sim \log \frac{\varphi(p-1)}{p-1} + \log x_t - \log \log x_t \sim \log x_t. \end{equation}
Combining the equivalences \eqref{I} and \eqref{II} completes the proof.
\end{proof}

\begin{proof}[Proof of Proposition~\ref{intermediate}]
Fix $\varepsilon \in (0,c-1)$, and choose (by Lemma \ref{asymp}) a positive integer $T$ such that if $t > T$ then
\begin{equation} (c-\varepsilon)t \log t < x_t < (c+\varepsilon) t \log t. \end{equation}
Consider $r > T$, and define
\begin{equation}
u_r = \log(x_t \cdots x_T) + (r-T) \log (c-\varepsilon) + \log
\left(\prod_{t=T+1}^r t \right)+ \log \prod_{t=T+1}^r \log t
\end{equation}
and 
\begin{equation}
v_r = \log(x_t \cdots x_T) + (r-T) \log (c+\varepsilon) + \log
\left(\prod_{t=T+1}^r t \right)+ \log \prod_{t=T+1}^r \log t.
\end{equation}
Using Stirling's approximation and Lemma \ref{asymp},
\begin{equation}
u_r \sim r \log (c-\varepsilon) + \log (r!) \sim r \log(c-\varepsilon) + r \log r \sim r \log r \sim \frac1c x_r,
\end{equation}
and similarly $v_r \sim \frac1c x_r$. Since $u_r < \log N_r < v_r$, the result now follows from the sandwich rule.
\end{proof}

\begin{proof}[Proof of Theorem~\ref{key}] 
By Lemma \ref{asymp} and Proposition \ref{intermediate},
\begin{equation} 
\frac{g^*(p,N_t)}{\log N_t} = \frac{x_{t+1}}{\log N_t} = \frac{x_{t+1}}{x_t} \cdot \frac{x_t}{\log N_t} \sim 
\frac{c(t+1) \log(t+1)}{ct \log t} \cdot \frac{x_t}{\log N_t} \to c,
\end{equation}
so it remains to show that $\limsup_{N \to \infty} \frac{g^*(p,N)}{\log N} \le c$. Fix $\varepsilon >0$.  For each positive integer $N$, choose (by our `worst case scenario' property) $t_N \ge 0$ such that $g^*(p,N_{t_N}) = g^*(p,N)$ and $N_{t_N} \le N$. Choose a positive integer $C$ such that if $t \ge C$  then $\frac{g^*(p,N_t)}{\log N_t} \le c + \varepsilon$, define the real number 
\begin{equation} M = \sup_{t > 0} \frac{g^*(p,N_t)}{\log N_t}, \end{equation}
and put
\begin{equation} K = \exp \frac{M \log N_C}{c+\varepsilon}.
\end{equation}
Let $N \ge K$. If $t_N \ge C$ then
\begin{equation}
\frac{g^*(p,N)}{\log N} \le \frac{g^*(p,N_{t_N})}{\log N_{t_N}} \le c + \varepsilon,
\end{equation}
while if $t_N < C$ then
\begin{align}
\frac{g^*(p,N)}{\log N} &= \frac{\log N_{t_N}}{\log N} \cdot \frac{g^*(p,N_{t_N})}{\log N_{t_N}} < 
\frac{\log N_{C}}{\log K} \cdot \frac{g^*(p,N_{t_N})}{\log N_{t_N}} \\
&= \frac{c+\varepsilon}M \cdot \frac{g^*(p,N_{t_N})}{\log N_{t_N}} \le c+\varepsilon,
\end{align}
which completes the proof since $\varepsilon >0$ was chosen arbitrarily.
\end{proof}

\section{A practical comparison}
\label{practical}

We know from Section \ref{asymptotics} that, for sufficiently large $N$,
\begin{equation} 
  g^*(p,N)^2 \le \frac1{12} N \prod_{\ell |N}
  \left(1+\frac1\ell\right)  \le \frac {k_2}{12} N \prod_{\ell |N}
  \left(1+\frac1\ell\right),  
\end{equation}
in the context of Theorem \ref{main}. In this section we describe how large $N$ has to be, given $p$, to ensure that
\begin{equation} \label{comp} 
  12g^*(p,N)^2 \le N \prod_{\ell | N} \left(1+\frac1\ell\right).
\end{equation}

Fix $p \ge 5$, and let $\hat N$ be minimal such that if $N \ge \hat N$ then the inequality \eqref{comp} holds. Our strategy will be to first establish a theoretical upper bound for $\hat N$, and then to determine $\hat N$ precisely using the software \emph{Sage}~\cite{Sage}. Our theoretical upper bound is $N_{r-1}$ in the following:
\begin{thm} \label{theoretical}
Assume GRH and let $p \ge 5$. Let $r = r(p)$ be minimal such that $N_{r-1} \ge 29.2032p^4(\log p)^4$, and suppose $N \ge N_{r-1}$. Then
\begin{equation} N \ge 12g^*(p,N)^2,\end{equation}
so in particular the inequality \eqref{comp} holds.
\end{thm}

To obtain this bound, we study the `worst case scenario' $N = N_{r-1}$. Our bound in this situation is:
\begin{prop} \label{WorstCaseII}
Assume GRH, let $p \ge 5$, and let $r$ be a positive integer such that
\begin{equation} N_{r-1} \ge 29.2032p^4(\log p)^4.\end{equation}
Then 
\begin{equation} \label{WorstComp} N_{r-1} \ge 12x_r^2.\end{equation}
\end{prop}

\subsection{The distribution of prime primitive roots modulo $p$, and more generally that of primes in arithmetic progression}

In pursuit of Proposition \ref{WorstCaseII}, we study the distribution of prime primitive roots modulo $p$. Specifically, we seek an explicit lower bound for the counting function. As this task is of intrinsic interest, we now indulge in a discussion that goes slightly beyond what is strictly necessary for our purposes. For a more comprehensive review, see the introduction of~\cite{Bach}. Many of the results in this section can be generalised to composite moduli. 

There are two main approaches to our task: (i) break the problem into $\varphi(p-1)$ primitive root residue classes modulo $p$ and study the distribution of primes in arithmetic progression, or (ii) specifically use the primitive root property. The approach (ii) is currently superior for deriving upper bounds for the least prime primitive root modulo $p$, for instance (assuming the Riemann hypothesis for all Hecke characters) Shoup~\cite{Shoup} uses sieve methods to provide the upper bound 
\begin{equation} O(r^4 (\log r +1)^4 (\log p)^2), \end{equation}
where $r$ is the number of distinct prime divisors of $p-1$;
note the discussion following~\cite[Corollary 3.1]{Martin}.

It is difficult to understand the distribution of such primes via the approach (ii), so we focus on (i). There are many classical asymptotic results, such as the prime number theorem for arithmetic progressions. For the least prime in an arithmetic progression $a \mod p$, where $p$ does not divide $a$, Linnik (see~\cite{Linnik1} and~\cite{Linnik2}) famously provided the upper bound 
\begin{equation} p^{O(1)},\end{equation}
and the exponent can be 5.2 unconditionally, if the bound is multiplied by a constant (see~\cite{Xylouris}). Conditional results are much stronger, and the conjectured upper bound is $p^2$ (see~\cite{Heath-Brown}).

Bach and Sorenson~\cite{Bach} derived an explicit version of Linnik's theorem:
\begin{thm} [see {\cite[Theorem 5.3]{Bach}}] \label{BS} Assume GRH. Let $a$ and $q$ be relatively prime positive integers. Then there exists $\ell \equiv a \mod q$ such that 
\begin{equation} \ell < 2(q \log q)^2. \end{equation}
\end{thm}

\begin{proof}[Summary of their approach]
For (Dirichlet) characters $\chi$ modulo $q$, real numbers $x >1$, and real numbers $\alpha$, put
\begin{equation} S(x,\chi) = \sum_{n<x} \Lambda(n) \chi(n) (n/x)^\alpha \log (x/n), \end{equation}
where $\Lambda$ is the von Mangoldt function. Let $a^{-1}$ denote the multiplicative inverse of $a$ modulo $q$. By orthogonality,
\begin{equation} \label{BSkey1} \sum_{\chi \mod q} \chi(a^{-1})S(x,\chi) = \varphi(q) \sum_{\substack{n<x \\ n \equiv a \mod q}} 
\Lambda(n) (n/x)^\alpha \log (x/n).\end{equation}
Suppose there exist no primes $\ell < x$ that are congruent to $a$ modulo $q$. Then
\begin{equation} \sum_{\chi \mod q} \chi(a^{-1})S(x,\chi) = p(x), \end{equation}
where $p(x)$ is the contribution of proper prime powers $n$ to the right hand side of equation \eqref{BSkey1}. For characters $\chi \mod q$, let $\hat \chi$ denote the primitive character induced by $\chi$. Then
\begin{equation} \label{BSkey2} 
\Big | \sum_{\chi \mod q} \chi(a^{-1})S(x,\hat \chi) \Big | \le |i(x)|+ p(x), \end{equation}
where 
\begin{equation} i(x) = \sum_{\chi \mod q} \chi(a^{-1}) (S(x,\hat \chi) - S(x,\chi)). \end{equation}
In~\cite[Subsection 4.1]{Bach}, tools from algebraic number theory and analytic number theory are used to bound $|i(x)|$ from above. In~\cite[Subsection 4.2]{Bach}, complex integration is used to estimate $|\sum_{\chi \mod q} \chi(a^{-1})S(x,\hat \chi)|$. In~\cite[Subsection 4.3]{Bach}, known estimates for a certain arithmetic function provide an upper bound for $p(x)$. In~\cite[Subsection 5.2]{Bach}, the cases $q \ge 1000$ and $q < 1000$ are considered separately. In the first case computer programs are used to choose $x$ and $\alpha$ so that the inequality \eqref{BSkey2} is invalidated, thereby proving that some prime $\ell < x$ is congruent to $a$ modulo $q$; the second case is handled by brute force.
\end{proof}
If further details are sought then~\cite[special case (1) on p362]{Bach0} and the proof of~\cite[Corollary~3.4]{Bach} describe our specific context within~\cite{Bach}. Note that~\cite[Theorem 5.3]{Bach} assumes the generalisation of the Riemann hypothesis to all Hecke $L$-functions, whereas the statement of Theorem \ref{BS} merely assumes it for Dirichlet $L$-functions. The stronger assumption is necessary for the more general results in~\cite{Bach}, but only GRH is needed for~\cite[Theorem 5.3]{Bach}. 
To justify this claim we use the notation of~\cite[Subsection 4.2]{Bach}, where Bach and Sorenson use the assumption for $\zeta_E$ and $L(\cdot, \hat \chi)$. The latter is a Dirichlet $L$-function, since $K=\bQ$ in our context, and the former is a product of Dirichlet $L$-functions (see~\cite[equation (2.2)]{Bach}), since for our purposes $E = \bQ(\zeta_q)$ is an abelian extension of $K=\bQ$, where $\zeta_q$ is a primitive $q$th root of unity.

The constant 2 appears to have been chosen for simplicity. Following the proof of~\cite[Theorem~5.3]{Bach}, but not rounding up until the end, and insisting that $q > 2$, the constant 2 can be improved to 1.56:
\begin{thm} \label{BS'}Assume GRH. Let $a$ and $q>2$ be relatively prime integers. Then there exists $\ell \equiv a \mod q$ such that 
\begin{equation} \ell < 1.56 (q \log q)^2. \end{equation}
\end{thm}
In fact the constant can be improved a little more (for $q>2$), but our theoretical bound for $\hat N$ will serve only as a ceiling for brute force computation, so we satisfy ourselves with the constant 1.56.

\subsection{A generalisation of Theorem \ref{BS'}}

We seek not the least prime in an arithmetic progression but the distribution of such primes, so we provide the following corollary:
\begin{cor} \label{dist}
Assume GRH. Let $a$ and $q>2$ be relatively prime integers, and let $t$ be a positive integer. Then
\begin{equation} \label{distbound} \pi_{a,q}(1.56t^2 q^{2t} (\log q)^2) \ge q^{t-1}.\end{equation}
\end{cor}
\begin{proof}
For each $s = 0,1,\ldots, q^{t-1}-1$, there exists $\ell \equiv a+sq \mod q^t$ such that
\begin{equation} \ell \le 1.56t^2 q^{2t} (\log q)^2,\end{equation}
by Theorem \ref{BS'}, since $(a+sq,q^t)=1$. These $\ell$ are distinct and congruent to $a$ modulo $q$.
\end{proof}

There are many ways in which to convert Corollary \ref{dist} into an explicit lower bound for $\pi_{a,q}(x)$ for all sufficiently large $x$; some are better asymptotically, while others do not require $x$ to be as large. Since our theoretical upper bound for $\hat N$ will serve merely as a ceiling for machine calculations, we have executed this fairly arbitrarily, and there may be other ways to improve our bound:
\begin{lemma} \label{AP}
Assume GRH. Let $a$ and $q \ge 5$ be relatively prime integers, and let
\begin{equation} x \ge 6.24 q^4 (\log q)^2.\end{equation}
Then
\begin{equation} \pi_{a,q}(x) > x^{1/9}.\end{equation}
\end{lemma}
\begin{proof}
Choose $t \ge 2$ such that
\begin{equation}
1.56t^2q^{2t} (\log q)^2 \le x < 1.56(t+1)^2q^{2(t+1)}(\log q)^2.
\end{equation}
By Corollary \ref{dist},
\begin{equation} 
  \pi_{a,q}(x) \ge \pi_{a,q}\left(1.56t^2 q^{2t} (\log q)^2\right) \ge q^{t-1}.
\end{equation}
Straightforward arithmetic confirms that $q^{t-1} > x^{1/9}$, completing the proof.
\end{proof}

By summing the bound \eqref{distbound} over the primitive root residue classes, we deduce:
\begin{cor} \label{upper}
Assume GRH, let $p>2$, and let $t$ be a positive integer. Then
\begin{equation}
x_{\varphi(p-1)p^{t-1}} \le 1.56t^2 p^{2t} (\log p)^2.
\end{equation}
\end{cor}

\subsection{Completion of the proof of Theorem \ref{theoretical}}

Now that we have an upper bound for the sequence $(x_r)$, we formulate a crude upper bound for the sequence $(N_r)$:
\begin{lemma} \label{lower} Let $p \ge 5$. Then
\begin{equation} N_{\varphi(p-1)} \ge (p+1)^{\varphi(p-1)/2}.\end{equation}
\end{lemma}
\begin{proof} Let $g_1, \ldots, g_{\varphi(p-1)}$ be integer representatives for the primitive root residue classes modulo $p$, with
\begin{equation}
1< g_1 < g_2 < \ldots < g_{\varphi(p-1)} < p.
\end{equation}
These come in pairs of inverses modulo $p$, and no $g_i$ can pair with itself because its order modulo $p$ is $p-1 > 2$. The product of each pair is at least $p+1$, so
\begin{equation}
N_{\varphi(p-1)} = x_1 \cdots x_{\varphi(p-1)} \ge g_1 \cdots g_{\varphi(p-1)} \ge (p+1)^{\varphi(p-1)/2}.
\end{equation}
\end{proof}

We show Proposition \ref{WorstCaseII} by first establishing a weaker bound:
\begin{prop}\label{WorstCaseI}
Assume GRH, let $p \ge 5$, and let $r$ be a positive integer such that
\begin{equation} N_{r-1} \ge 467.2512p^8(\log p)^4.\end{equation}
Then $N_{r-1} \ge 12 x_r^2$.
\end{prop}
\begin{proof}
Proof by contradiction: assume that $N_{r-1} < 12 x_r^2$. Then 
\begin{equation} \label{first} x_r > 6.24 p^4 (\log p)^2, \end{equation}
so Lemma \ref{AP} gives
\begin{equation} \label{sumthis} \pi_{a,p}(x_r) > x_r^{1/9} \end{equation}
for all integers $a$ that are not divisible by $p$. Since $r$ is the number of prime primitive roots modulo $p$ that are less than or equal to $x_r$, summing the inequality \eqref{sumthis} over all primitive root residue classes $a$ modulo $p$ yields
\begin{equation} r> \varphi(p-1)x_r^{1/9}.\end{equation}
Now
\begin{equation} \label{second} 
  N_{r-1} < 12 x_r^2 < 12 \left( \frac r {\varphi(p-1)} \right)^{18}.
\end{equation}

Specialising $t=2$ in Corollary \ref{upper} yields
\begin{equation} x_{p\varphi(p-1)} \le 6.24p^4(\log p)^2,\end{equation}
which together with the inequality \eqref{first} implies that $r > p \varphi(p-1)$. Induction shows that if $r>44$ then $N_{r-1} \ge 12(0.5r)^{18}$ (use the product of the first $r-1$ primes as a crude lower bound for $N_{r-1}$), which would contradict the inequality \eqref{second}. Hence $p \varphi(p-1) < r \le 44$, so $p=5,7$. In each of these cases $10 <r \le 44$ and $N_{10} > 12x_{44}^2$ (by computer check), completing the proof.
\end{proof}

Finally we prove Proposition~\ref{WorstCaseII} and Theorem~\ref{theoretical}.
\begin{proof}[Proof of Proposition~\ref{WorstCaseII}]
First assume that $p \ge 71$. In this case it is easy to show, by considering cases, that $\varphi(p-1) \ge 24$. Specialising $t=1$ in Corollary \ref{upper} yields
\begin{equation} \label{t1}
x_{\varphi(p-1)} \le 1.56p^2 (\log p)^2,
\end{equation}
so the result follows immediately if $r \le \varphi(p-1)$. However, if $r > \varphi(p-1)$ then, using Lemma~\ref{lower},
\begin{equation} N_{r-1} \ge N_{\varphi(p-1)} \ge (p+1)^{\varphi(p-1)/2} \ge (p+1)^{12} \ge 467.2512p^8 (\log p)^4, \end{equation}
whereupon the result follows from Proposition \ref{WorstCaseI}.

For each $p$ with $5 \le p < 71$, there are very few values of $r$ for which 
\begin{equation} 29.2032p^4(\log p)^4 \le N_{r-1} < 467.2512p^8(\log p)^4,\end{equation}
so we computer check these cases and apply Proposition \ref{WorstCaseI} otherwise.
\end{proof}

\begin{proof}[Proof of Theorem~\ref{theoretical}]
Let $g^*(p,N) = x_s$, so that $N \ge N_{s-1}$, and put $t = \max(r,s)$. Then 
\begin{equation} N_{t-1} \ge N_{r-1} \ge 29.2032 p^4 (\log p)^4 \end{equation}
so, by Proposition \ref{WorstCaseII}, $N_{t-1} \ge 12 x_t^2$. Now
\begin{equation} N \ge N_{t-1} \ge 12 x_t^2 \ge 12 x_s^2 = 12 g^*(p,N)^2.\end{equation}
\end{proof}

\subsection{Computation of $\hat N$ given $p$}

Henceforth, let $r$ be as in Theorem \ref{theoretical}, and assume GRH. Now that we have a theoretical upper bound for $\hat N$, it is not too difficult to write a program that, given $p$, will compute $\hat N$ exactly. Still, it would be awfully slow to test the inequality \eqref{comp} for every $N < N_{r-1}$, so we shall describe an economising manoeuvre based on the following observation:
\begin{lemma} \label{obs}
Let $t$ be a positive integer, and suppose that $N \ge 12x_t^2$ is such that the inequality~\eqref{comp} does not hold. Then $N_t$ divides $N$. 
\end{lemma}
\begin{proof} The hypotheses imply that $g^*(p,N) > x_t$, so $N_t$ divides $N$. \end{proof}

So we only need to test the inequality \eqref{comp} for $N \le 12 x_1^2$ and for multiples of $N_t$ in the range 
\begin{equation} \label{range} 
  \left[12x_t^2, 12x_{t+1}^2\right) 
\end{equation}
($t=1,2,\ldots,r-2$), since Lemma \ref{obs} and Theorem \ref{theoretical} imply that if $N \ge 12 x_{r-1}^2$ then the inequality \eqref{comp} holds. 

There is a reasonable upper bound \eqref{t1} for $x_{\varphi(p-1)}$, and hence for $x_1$, however in practice $x_1$ is very small. Moreover, for each $t$ there are very few (if any) multiples of $N_t$ in the range~\eqref{range}. Consequently, we have an extremely efficient method for determining $\hat N$ given $p$, and we could easily have done so for much larger $p$ than discussed below. By running the program  we conclude as follows:

\begin{prop} For $p \ge 5$, the inequality \eqref{comp} holds if $p < p^*$ and $N \ge N^*$ for the following pairs $(p^*,N^*)$: 
\begin{align*} &(4243,121424) \qquad &(2791,81550) \qquad &(691,48204) \qquad &(271,44158) \\
&(199,38858) \qquad &(151,24796) \qquad &(43,9049) \qquad &(19,5853).
\end{align*}
In particular, in any of these cases the bound in Theorem \ref{main} becomes
\begin{equation}
\frac{k_2}{12}N \prod_{\ell |N}\left(1+\frac1 \ell\right).
\end{equation}
\end{prop}

These are best possible bounds for $\hat N$, since for each $p$ we computed $\hat N$ exactly. One might wonder why $\hat N$ is so large. Indeed $g^*(p,N)$ is typically very small, however there are some values (small multiples of the $N_t$) for which $g^*(p,N)$ is somewhat large, which can mean that the inequality \eqref{comp} suddenly fails.

\section{The special cases $p=2$ and $p=3$}
\label{2and3}

As the considerations in this section are not crucial to the main point of the
paper, we do not recall here the algebro-geometric definition of modular
forms due to Deligne and Katz.  The interested reader is invited to
consult~\cite{Katz} or~\cite{Gross}.

For any prime $p$, the Hasse invariant $A_p$ is a Katz modular form (mod $p$)
of level one and weight $p-1$, with $q$-expansion
\begin{equation*}
  A_p(q)=1.
\end{equation*}
As recalled in Section~\ref{proof}, if $p\geq 5$ then $A_p$ can be obtained as
the reduction modulo $p$ of the Eisenstein series $E_{p-1}$.  We say that
$E_{p-1}$ is a lifting of $A_p$ to characteristic zero.  If $p<5$, we can
still lift $A_p$ to a form in characteristic zero, at the expense of
increasing the level.  We will use the following two results of Katz:
\begin{thm}[see~{\cite[Theorem 1.7.1]{Katz}}]
  \label{katz1}
  Let $k$ and $N$ be positive integers such that either ($k=1$ and $3\leq
  N\leq 11$) or ($k\geq 2$ and $N\geq 3$).  Let $p$ be a prime not dividing
  $N$.  Then every modular form (mod $p$) of weight $k$ and level $\Gamma(N)$
  can be lifted to characteristic zero.
\end{thm}
\begin{thm}[see~{\cite[Theorem 1.8.1]{Katz}}]
  \label{katz2}
  Let $k$ be a positive integer and let $p\neq 2$ be a prime.  Every modular
  form (mod $p$) of weight $k$ and level $\Gamma(2)$ can be lifted to
  characteristic zero.
\end{thm}

\subsection{The case $p=3$}
\begin{itemize}
  \item If $N$ is a power of $3$, we can use Theorem~\ref{katz2} to lift $A_3$
    to $\tilde{A}_3$:
    \begin{equation*}
      \xymatrixcolsep{0pt}
      \xymatrix{
      A_3 & \in & M_2(SL_2(\bZ);\overline{\bF}_3)
      & \subset & 
      M_2(\Gamma(2);\overline{\bF}_3) \ar@{~>}[d]\\
      & & \tilde{A}_3 & \in & M_2(\Gamma(2);\overline{\bZ})
      & \subset &
      M_2(\Gamma_0(2),\text{triv};\overline{\bZ}).
    }
    \end{equation*}
    Going through the proof in
    Section~\ref{proof} with $E_{p-1}$ replaced by $\tilde{A}_3$, we have
    $f^\prime=\tilde{A}_3^r f\in
    M_{k_2}(\Gamma_0(2N),\chi;\overline{\bZ})$, so we must use the
    Sturm bound for $\Gamma_0(2N)$.  Therefore the inequality in
    Theorem~\ref{main} must be replaced by
    \begin{equation}
    n \le \max\left\{ g^*(p,N)^2, \frac{k_2}{12}[SL_2(\bZ):\Gamma_0(2N)]
    \right\}
    \end{equation}
  \item If $N$ is divisible by $2$, the same process as in the previous part
    gives us the lifting 
    $\tilde{A}_3\in M_2(\Gamma_0(2),\text{triv};\overline{\bZ})$.  However,
    since $2$ divides $N$, we obtain the exact same inequality as in
    Theorem~\ref{main}.
  \item If $N$ is divisible by a prime $p_0\notin\{2, 3\}$, we can use 
    Theorem~\ref{katz1} to lift 
    \begin{equation*}
      \xymatrixcolsep{0pt}
      \xymatrix{
      A_3 & \in & M_2(SL_2(\bZ);\overline{\bF}_3)
      & \subset & 
      M_2(\Gamma(p_0);\overline{\bF}_3) \ar@{~>}[d]\\
      & & \tilde{A}_3 & \in & M_2(\Gamma(p_0);\overline{\bZ})
      & \subset &
      M_2(\Gamma_0(p_0),\text{triv};\overline{\bZ}).
    }
    \end{equation*}
    Since $p_0$ divides $N$, we again obtain the same inequality as in 
    Theorem~\ref{main}.
\end{itemize}

\subsection{The case $p=2$}
\begin{itemize}
\item If $N$ is not
divisible by $5$, $7$ or $11$, use Theorem~\ref{katz2} to lift $A_2$ to
$\tilde{A}_2$:
\begin{equation*}
  \xymatrixcolsep{0pt}
  \xymatrix{
  A_2 & \in & M_1(SL_2(\bZ);\overline{\bF}_2)
  & \subset & 
  M_1(\Gamma(5);\overline{\bF}_2) \ar@{~>}[d]\\
  & & \tilde{A}_2 & \in & M_1(\Gamma(5);\overline{\bZ})
  & \subset &
  M_1(\Gamma_0(5),\text{triv};\overline{\bZ}).
}
\end{equation*}
The inequality in
Theorem~\ref{main} must then be replaced by
\begin{equation}
n \le \max\left\{ g^*(p,N)^2, \frac{k_2}{12}[SL_2(\bZ):\Gamma_0(5N)] \right\}
\end{equation}
\item If $N$ is divisible by $p_0\in\{5, 7, 11\}$, use Theorem~\ref{katz1}
  to lift
  \begin{equation*}
    \xymatrixcolsep{0pt}
    \xymatrix{
    A_2 & \in & M_1(SL_2(\bZ);\overline{\bF}_2)
    & \subset & 
    M_1(\Gamma(p_0);\overline{\bF}_2) \ar@{~>}[d]\\
    & & \tilde{A}_2 & \in & M_1(\Gamma(p_0);\overline{\bZ})
    & \subset &
    M_1(\Gamma_0(p_0),\text{triv};\overline{\bZ}).
  }
  \end{equation*}
  Since $p_0$ divides $N$, we
  get the same inequality as in Theorem~\ref{main}.
\end{itemize}

We summarise our findings in Table~\ref{table:ineq}.

\begin{table}[h]
\begin{tabular}{lll}
  \toprule
  Prime & Level & Inequality in Theorem~\ref{main}\\
  \midrule
  $p\geq 5$ & $N\geq 1$ \\
  $p=3$ & $N\neq 3^a$, some $a$ &
    $n \le \max\left\{ g^*(p,N)^2, \frac{k_2}{12}[SL_2(\bZ):\Gamma_0(N)]
    \right\}$
    \\
  $p=2$ & $N$ divisible by $5$, $7$ or $11$\\
  \midrule
  $p=3$ & $N=3^a$, some $a$ & 
    $n \le \max\left\{ g^*(p,N)^2, \frac{k_2}{12}[SL_2(\bZ):\Gamma_0(2N)]
    \right\}$
  \\
  \midrule
  $p=2$ & $N$ not divisible by $5$, $7$ or $11$ &
    $n \le \max\left\{ g^*(p,N)^2, \frac{k_2}{12}[SL_2(\bZ):\Gamma_0(5N)]
    \right\}$\\
  \bottomrule
\end{tabular}
\caption{Inequalities obtained for the various combinations of $p$ and $N$}
\label{table:ineq}
\end{table}


\begin{thebibliography}{99}
\bibitem{Bach0} E. Bach, \emph{Explicit bounds for primality testing and related problems},
Mathematics of Computation \textbf{55} (1990), no. 191, 355--380.
\bibitem{Bach} E. Bach and J. Sorenson, \emph{Explicit bounds for primes in
  residue classes}, Mathematics of Computation \textbf{65} (1996), 1717--1735.
\bibitem{BorevichShafarevich} Z. I. Borevich and I. R. Shafarevich,
\emph{Number Theory}, Academic Press, 1966.
\bibitem{Diamond} F. Diamond and J. Shurman, \emph{A first course in modular forms}, Springer, 2005.
\bibitem{Ghitza} A. Ghitza, \emph{Distinguishing Hecke eigenforms},
International Journal of Number Theory \textbf{7} (2011), 1247--1253.
\bibitem{GH} D. Goldfeld and J. Hoffstein, \emph{On the number of terms that
  determine a modular form}, Contemp. Math., AMS, \textbf{143} (1993), 385--393.
\bibitem{Gross} B. Gross, \emph{A tameness criterion for Galois
  representations associated to modular forms (mod $p$)}, Duke Mathematical
Journal \textbf{61} (1990), no. 2, 445--517.
\bibitem{Heath-Brown} D. R. Heath-Brown, \emph{Zero-free regions for
Dirichlet L-functions, and the least prime in an arithmetic progression},
Proceedings of the London Mathematical Society \textbf{64} (1992), no. 3,
265--338.
\bibitem{Katz} N. Katz, \emph{$p$-adic properties of modular schemes and
  modular forms}, Lecture Notes in Mathematics \textbf{350} (1973),
Springer-Verlag, 69--190.
\bibitem{Kilford} L. Kilford, \emph{Modular Forms: A classical and computational introduction}, Imperial College Press, 2008.
\bibitem{Kohnen} W. Kohnen, \emph{On Fourier coefficients of modular forms of
  different weights}, Acta Arithmetica \textbf{113} (2004), no. 1, 57--67.
\bibitem{Kowalski} E. Kowalski, \emph{Variants of recognition problems for
  modular forms}, Archiv der Mathematik \textbf{84} (2005), no. 1, 57--70.
\bibitem{Linnik1} U. V. Linnik, \emph{On the least prime in an arithmetic
  progression. I. The basic theorem}, Rec. Math. (Mat. Sbornik) N.S.,
\textbf{15 (57)} (1944), 139--178.
\bibitem{Linnik2} U. V. Linnik, \emph{On the least prime in an arithmetic
  progression. II. The Deuring-Heilbronn phenomenon}, Rec. Math. (Mat.
Sbornik) N.S., \textbf{15 (57)} (1944), 347--368.
\bibitem{Martin} G. Martin, \emph{The least prime primitive root and the
shifted sieve}, Acta Arithmetica \textbf{80} (1997), no. 3, 277--288.
\bibitem{Murty} M. Ram Murty, \emph{Congruences between modular forms}, London
Mathematical Society Lecture Note Series \textbf{247} (1997), Cambridge
University Press, 309--320.
\bibitem{Shoup} V. Shoup, \emph{Searching for primitive roots in finite fields}, Math. Comp. \textbf{58} (1992), 369--380.
\bibitem{Stein} W. Stein, \emph{Modular Forms, a Computational Approach}, Graduate Studies in Mathematics \textbf{79}, American Mathematical Society, 2007. With an appendix by Paul E. Gunnells.
\bibitem{Sage} W. Stein et al., Sage Mathematics Software (Version 5.1), The Sage Development Team, 2012, \url{http://www.sagemath.org}.
\bibitem{Sturm} J. Sturm, \emph{On the congruence of modular forms}, Lecture
Notes in Mathematics \textbf{1240} (1987), Springer-Verlag, 275--280.
\bibitem {Xylouris} T. Xylouris, \emph{On Linnik's constant}, Acta
Arithmetica \textbf{150} (2011), no. 1, 65--91.
\end{thebibliography}
\end{document}